\documentclass[reqno,makeidx,12pt]{amsart}
\usepackage[T1]{fontenc}
\usepackage[utf8]{inputenc}  

\usepackage{graphicx}
\usepackage{latexsym}
\usepackage{amstext}
\usepackage {amsmath}
\usepackage {amsfonts}
\usepackage {amssymb}
\usepackage {amsthm}
\usepackage {bbm}
\usepackage{enumerate}
\usepackage{array}
\usepackage{comment}
\usepackage{xspace}
\usepackage[bookmarks=true]{hyperref}
\usepackage{enumerate}
\ifx\cs\documentclass \footheight 12pt \fi \footskip 30pt 
\DeclareMathOperator\id{Id}

\DeclareMathOperator{\dive}{div}

\DeclareMathOperator{\A}{A}
\DeclareMathOperator{\tr}{tr}

\newcommand{\R}{\ensuremath{\mathbb{R}}}
\newcommand{\Rn}{\ensuremath{{\mathbb{R}^n}}}

\newcommand{\Ha}{\ensuremath{\mathcal{H}}}
\def\cringle{\mathaccent"7017 }
\newcommand{\Atf}{\ensuremath{\cringle{\A}}}

\newcommand{\ar}{ar}
\newcommand{\vol}{vol}

\def\R{\mathbb R}

\def\W{\mathcal{W}}
\def\I{\mathcal{I}}
\def\M{\mathcal{M}}

\theoremstyle{note}
\newtheorem{theorem}{Theorem}
\theoremstyle{definition}
\newtheorem{remark}{Remark}
\begin{document}
\title[]{Control of the isoperimetric deficit by the Willmore deficit}
\author{Matthias R{\"o}ger}
\address{Matthias R{\"o}ger, Technische Universit\"at Dortmund,
Fakult\"at für Mathematik,
Vogelpothsweg 87,
D-44227 Dortmund}
\email{matthias.roeger@tu-dortmund.de}

\author{Reiner Sch{\"a}tzle}
\address{Reiner Sch{\"a}tzle, Eberhard-Karls-Universit\"at T\"ubingen,
Mathematisches Institut,
Auf der Morgenstelle 10,
D-72076 T\"ubingen}
\email{schaetz@everest.mathematik.uni-tuebingen.de}



\date{\today}

\begin{abstract}
In the class of smoothly embedded surfaces of sphere type we prove that the isoperimetric deficit can be controlled by the Willmore deficit.
 \end{abstract}

\maketitle
\label{sec:intro}
Let us consider the class $\M$ of smoothly embedded surfaces $\Sigma\subset\R^3$ of sphere type with enclosed inner region $\Omega_\Sigma\subset\R^3$, and let us denote by  $\ar(\Sigma)$ and $\vol(\Omega_\Sigma)$ the two-dimensional Hausdorff measure of $\Sigma$ and the three-dimensional Lebesgue measure of $\Omega_\Sigma$, respectively. For $\Sigma\in\M$ we define the isoperimetric ratio
\begin{gather}
	\I(\Sigma)\,:=\, \frac{\ar(\Sigma)}{\vol(\Omega_\Sigma)^{\frac{2}{3}}}, \label{eq:isop}
\end{gather}
and the Willmore energy
\begin{gather}
	\W(\Sigma)\,:=\, \frac{1}{4}\int_\Sigma |H|^2\,d\Ha^2, \label{eq:willm}
\end{gather}
where $H$ denotes the mean curvature of the surface $\Sigma$. Both functionals are invariant under dilations and translations; round spheres are in both cases the unique minimizers, in particular
\begin{alignat*}{2}
	\I(\Sigma)\,&\geq\, \I(S^2)\,&&=\, (6\sqrt{\pi})^{\frac{2}{3}},\\
	\W(\Sigma)\,&\geq\, \W(S^2)\,&&=\, 4\pi
\end{alignat*}
for all $\Sigma\in\M$. We consider for any such $\Sigma$ and both functionals the corresponding \emph{deficits}, that is the difference from the optimal value. Our main result is the following control of the isoperimetric deficit by the Willmore deficit.
\begin{theorem}\label{thm:main}
For all $c_0>0$ there exists a universal constant $C>0$ such that
\begin{gather}
	\I(\Sigma) - \I(S^2) \,\leq\, C\Big(\W(\Sigma)-\W(S^2)\Big) \label{eq:control}
\end{gather}
for all $\Sigma\in\M$ with $\I(\Sigma)-\I(S^2)\leq c_0$.
\end{theorem}
Umbilical surfaces $\Sigma\in \M$ are by a classical Theorem of Codazzi round spheres. Theorem \ref{thm:main} can be also been seen as a quantitative version of this statement: An equivalent formulation of \eqref{eq:control} is that
\begin{gather*}
	\I(\Sigma) - \I(S^2) \,\leq\, C\frac{1}{4}\int_\Sigma (\kappa_1-\kappa_2)^2\,d\Ha^2,
\end{gather*}
where $\kappa_1,\kappa_2$ denote the prinicipal curvatures of $\Sigma$. This equivalence follows directly from \eqref{eq:atf} below.\\
We include some comments on the optimality of \eqref{eq:control}. By an explicit example we show below that the linear growth in the right-hand of \eqref{eq:control} with respect to the Willmore deficit is optimal. Next we observe that  \eqref{eq:control} cannot hold for arbitrary $\Sigma\in\M$: In recent work by Schygulla \cite{Schy11} it is shown that for arbitrarily prescribed isoperimetric ratio $(6\sqrt{\pi})^{\frac{2}{3}}\leq\sigma<\infty$ the minimial Willmore energy in the class  $\M_\sigma:=\{\Sigma\in\M\,:\, \I(\Sigma)=\sigma\}$ is attained and is strictly below $8\pi$. Finally, the isoperimetric deficit is not estimated from below by the Willmore deficit, since the Willmore functional can be made arbitrarily large by variations of the sphere that are small in $C^1$ but large in $C^2$, whereas the isoperimetric ratio does not change substantially.

Estimates from below for the isoperimetric deficit have however been proved  previously in terms of more direct concepts of distance from balls. Bernstein \cite{Bern05} and Bonnesen \cite{Bonn24} already considered in two space dimensions the \emph{asymmetry index}
\begin{gather*}
	A(\Omega)\,:=\,\inf\big\{\Omega\Delta B(x_0,r)\,:\, x_0\in \R^n, r>0, \vol(B(x_0,r))=\vol(\Omega)\big\}
\end{gather*}
and proved a lower bound for the isoperimetric deficit in terms of this quantity. Figalli, Maggi, and Pratelli \cite{FiMP10} generalized and sharpened their results and proved the existence of a constant $C=C(n)$ such that
\begin{gather*}
	A(\Omega)\,\leq\, C \big(\I(\partial\Omega)-\I(S^{n-1})\big)^{\frac{1}{2}}
\end{gather*}
for every measureable set $\Omega\subset\Rn$ with $0<\vol(\Omega)<\infty$. As the left-hand side is not sensitive with respect to variations that are small in $C^0$ but large in $C^1$ whereas the right-hand side is, an opposite inequality of this type cannot be true. In this sense the optimality of the ball with respect to the isoperimetric ratio is in between the corresponding optimality properties with respect to the  asymmetry index and with respect to the Willmore energy.

Before we start with the proof of Theorem \ref{thm:main} we fix some notations and collect some results that we will use below. 
Let $\Sigma\subset\R^3$ denote a smoothly embedded hypersurface in $\R^3$ of sphere type. We denote by $A$ the second fundamental form of $\Sigma$ and by $\Atf$ the trace-free part of the second fundamental form,
\begin{gather*}
	\Atf(x)\,=\, A(x)-\frac{\tr{A(x)}}{2} \id\,=\, A(x)-\frac{1}{2}H(x)\id\qquad\text{ for }x\in\Sigma.
\end{gather*}
If we denote by $\kappa_1(x)$ and $\kappa_2(x)$ the principal curvatures of $A(x)$ we obtain
\begin{gather*}
	2|\Atf|^2\,=\, \kappa_1^2 +\kappa_2^2 -2\kappa_1\kappa_2 \,=\, H^2 -4 K,
\end{gather*}
where $K$ denotes the Gauss curvature of $\Sigma$. The Gauss--Bonnet Theorem then implies that
\begin{gather}
	\W(\Sigma)-\W(S^2)\,=\, \frac{1}{4}\int_\Sigma H^2 \,d\Ha^2 -4\pi \,=\, \frac{1}{2}\int_\Sigma |\Atf|^2\,d\Ha^2. \label{eq:atf}
\end{gather}

\begin{proof}[Proof of Theorem \ref{thm:main}]
In a first step we show the existence of a universal constant $C>0$ such that for all $\Sigma$ with $\W(\Sigma)\leq 6\pi$ and $\ar(\Sigma)=4\pi$
\begin{gather}
	\Big(\vol(B)-\vol(\Omega_\Sigma)\Big) \,\leq\, C\big(\W(\Sigma)-\W(S^2)\big), \label{eq:est-vol}
\end{gather}
where $B\subset\R^3$ denotes the unit ball. To prove this inequality we first apply a theorem by de Lellis and Müller \cite{DeMu05}: possibly after a suitable translation of $\Sigma$, there exists a conformal parametrization $\psi:S^2\to\Sigma$ such that
\begin{gather}
	\|\psi -\id\|_{W^{2,2}(S^2)}\,\leq\, C\|\Atf\|_{L^2(\Sigma)}. \label{eq:LeMu}
\end{gather}
We denote by $g$ the conformal factor of the induced area element, that is $\psi_\sharp \big(\Ha^2\lfloor \Sigma\big)=g\Ha^2\lfloor S^2$ and set $H_\psi(x)=H_\Sigma(\psi(x))$ for $x\in S^2$. From $\ar(\Sigma)=4\pi$ we obtain that
\begin{gather}
	\int_{S^2} g(x)\,d\Ha^2(x) \,=\, \int_\Sigma 1\,d\Ha^2\,=\, \int_{S^2} 1\,d\Ha^2.  \label{eq:ar4pi}
\end{gather}
For the volume of $\Sigma$ we deduce from the Divergence Theorem that
\begin{gather}
	3\vol(\Omega_\Sigma)\,=\, \int_{\Omega_\Sigma}  \dive x\,dx\,=\, \int_\Sigma x\cdot \nu(x)\,d\Ha^2(x),
\end{gather}
where $\nu$ denotes the outer normal of $\Omega_\Sigma$. By rewriting the last integral in terms of the parametrization $\psi$ we obtain
\begin{align}
	3\Big(\vol(B)-\vol(\Omega_\Sigma)\Big)\,&=\, \int_{S^2} \Big(1-\psi(x)\cdot \nu(x)g(x)\Big)\,d\Ha^2(x) \notag\\
	&=\, \int_{S^2} \Big(1-\psi(x)\cdot \nu(x)\Big)g(x)\,d\Ha^2(x), \label{eq:pf1}
\end{align}
where we have used \eqref{eq:ar4pi}.
We further compute that
\begin{align*}
	1-\psi\cdot\nu \,=\, &\frac{1}{2}|\psi-\nu|^2 - \frac{1}{2}(|\psi|^2-1) \\
	=\, &\frac{1}{2}|\psi-\nu|^2-\frac{1}{2}|\psi+\frac{1}{2}H_\psi|^2 +\frac{1}{2}\psi\cdot H_\psi +\frac{1}{8}H_\psi^2 +\frac{1}{2}.
\end{align*}
Using that
\begin{align*}
	\int_{S^2} \frac{1}{2}\psi\cdot H_\psi g\,d\Ha^2\,
	=\, &\int_\Sigma \frac{1}{2} x\cdot H_\Sigma(x)\,d\Ha^2(x)\\
	=\,& -\int_\Sigma \frac{1}{2}\text{div}_{\text{tan}}x\,d\Ha^2(x)
	\,=\, -\ar(\Sigma)\,=\, -\int_{S^2} \Big(\frac{1}{2}+\frac{1}{2}g\Big)\,d\Ha^2
\end{align*}
we therefore deduce from \eqref{eq:pf1} that
\begin{align}
	&3\Big(\vol(B)-\vol(\Omega_\Sigma)\Big)\notag\\
	=\, &\int_{S^2} \Big(\frac{1}{2}|\psi-\nu|^2g-\frac{1}{2}|\psi+\frac{1}{2}H_\psi|^2g +\frac{1}{2}\psi\cdot H_\psi g +\frac{1}{8}H_\psi^2g +\frac{1}{2}g\Big) \,d\Ha^2 \notag\\
	=\, & \int_{S^2} \Big(\frac{1}{2}|\psi-\nu|^2g-\frac{1}{2}|\psi+\frac{1}{2}H_\psi|^2g +\frac{1}{8}H_\psi^2g -\frac{1}{2}\Big)\,d\Ha^2 \notag\\
	\leq\, & \int_{S^2} \frac{1}{2}|\psi-\nu|^2g \,\,d\Ha^2+\frac{1}{2}\Big(\W(\Sigma)-\W(S^2)\Big) \notag\\
	\leq\, & \int_{S^2} \Big(|\psi(x)-x|^2 + |x-\nu(x)|^2 \Big)g(x)\,d\Ha^2(x)+\frac{1}{2}\Big(\W(\Sigma)-\W(S^2)\Big)\notag\\
	\leq\, & C\|\Atf\|_{L^2(\Sigma)}^2 +\frac{1}{2}\Big(\W(\Sigma)-\W(S^2)\Big)
\end{align}
by \eqref{eq:LeMu}. With \eqref{eq:atf}  inequality \eqref{eq:est-vol} follows.

We now choose $\delta:=\min\{\frac{2}{3C},2\}$ with $C>0$ from \eqref{eq:est-vol}. We then have for all $\Sigma\in\M$ with $\ar(\Sigma)=4\pi$ and $\W(\Sigma)< (4+\delta)\pi$ that $\vol(\Omega_\Sigma)>\frac{2\pi}{3}$ holds. Since $\vol(\Omega_\Sigma)\leq \vol(B)\,=\, \frac{4\pi}{3}$ and since $a^{-2/3}-b^{-2/3}\leq \frac{2}{3}a^{-5/3}(b-a)$ for all $0<a<b$ we deduce that
\begin{align}
	\I(\Sigma)-\I(S^2)\,&=\, \frac{4\pi}{\vol(\Omega_\Sigma)^{\frac{2}{3}}}-\frac{4\pi}{\vol(B)^{\frac{2}{3}}}\,\leq\, \frac{8\pi}{3}\Big(\frac{2\pi}{3}\Big)^{-\frac{5}{3}}\Big(\vol(B)-\vol(\Omega_\Sigma)\Big) \notag\\
	&\leq\, C_1\Big(\W(\Sigma)-\W(S^2)\Big) \label{eq:control1}
\end{align}
for all $\Sigma$ with $\W(\Sigma)<(4+\delta)\pi$ and $\ar(\Sigma)=4\pi$ by inequality \eqref{eq:est-vol}. Since both sides of \eqref{eq:control1} are invariant under dilations this proves
\begin{align}
	\I(\Sigma)-\I(S^2)\,&\leq\, C_1\Big(\W(\Sigma)-\W(S^2)\Big) \label{eq:control2}
\end{align}
for all $\Sigma\in \M$ with $\W(\Sigma)<(4+\delta)\pi$. On the other hand, for all $\Sigma\in\M$ with $\W(\Sigma)\geq (4+\delta)\pi$ and $\I(\Sigma)-\I(S^2)\leq c_0$ we have
\begin{gather*}
	\I(\Sigma)-\I(S^2) \,\leq\,  \frac{c_0}{\delta\pi} \Big(\W(\Sigma)-\W(S^2)\Big).
\end{gather*}
Setting $C=\max\{C_1, \frac{c_0}{\delta\pi}\}$ this proves together with \eqref{eq:control2} that \eqref{eq:control} holds for all $\Sigma\in\M$ with $\I(\Sigma)-\I(S^2)\leq c_0$.
\end{proof}
\begin{remark}
The optimality of \eqref{eq:control} with respect to the (linear) growth rate in  the right-hand side can be easily seen by evaluating a specific perturbation of the unit sphere. Consider for example the ellipsoids
\begin{gather*}
	E(r)\,=\, \Big\{ \begin{pmatrix}
					x\cos\varphi \\ x\sin\varphi\\ r f(x)
				\end{pmatrix}
				\,:\, \varphi\in [0,2\pi),\, x\in [0,1]\Big\}
\end{gather*}
for $f(x)=\sqrt{1-x^2}$. We then have $E(1)=S^2$. A direct computation shows that with $q(x):= \sqrt{1+r^2(f'(x))^2}$ for $r>1$
\begin{align*}
	\int_{E(r)} H^2 \,d\Ha^2\,&=\, 4\pi \int_{0}^1 xq(x)r^2\Big(\frac{f''(x)}{q(x)^3}+\frac{f'(x)}{xq(x)}\Big)^2\,dx\\
	&=\, 4\pi\Big(\frac{7r^2+2}{3r^2}+\frac{r^2}{\sqrt{r^2-1}}\Big(\frac{\pi}{2}-\arctan\frac{1}{\sqrt{r^2-1}}\Big)\Big).
\end{align*}
For the corresponding volume of the region $\Omega(r)$ enclosed by $E(r)$ we obtain that
\begin{gather*}
	\vol(\Omega(r))\,=\, \frac{4\pi}{3} r
\end{gather*}
and for the area and isoperimetric ratio
\begin{align*}
	\ar(E(r))\,&=\, 4\pi\int_0^1 xq(x)\,dx\,=\, 2\pi \Big(1+\frac{r^2}{\sqrt{r^2-1}}\arcsin\Big(\frac{\sqrt{r^2-1}}{r}\Big)\Big)\\
	\I(E(r))\,&=\, (6\sqrt{\pi})^{\frac{2}{3}}\frac{1}{2}\Big(r^{-\frac{2}{3}}+\frac{r^\frac{4}{3}}{\sqrt{r^2-1}}\arcsin\Big(\frac{\sqrt{r^2-1}}{r}\Big)\Big).
\end{align*}
Using that $\frac{\pi}{2}-\arctan\frac{1}{\sqrt{r^2-1}}=\arcsin\big(\frac{\sqrt{r^2-1}}{r}\big)$ a few direct computations yield that
\begin{gather*}
	\lim_{r\searrow 1} \frac{\W(E(r))}{\I(E(r))}\,=\, 6\Big(\frac{16\pi}{3}\Big)^\frac{2}{3}.
\end{gather*}
\end{remark}

\end{document}